\documentclass{amsart}

\usepackage{graphicx}
\usepackage{amsthm}
\usepackage{amsmath, amssymb}
\usepackage{color}
\usepackage{accents}
\usepackage{subfigure}
\usepackage{url}
\newcommand{\R}{{\mathbb{R}}}

\newcommand{\A}{{\mathcal{A}}}
\newcommand{\C}{{\mathcal{C}}}

\DeclareMathOperator{\Int}{Int}

\theoremstyle{plain}
\newtheorem{thm}{Theorem}[section]
\newtheorem{cor}[thm]{Corollary}
\newtheorem{lem}[thm]{Lemma}

\theoremstyle{definition}
\newtheorem{dfn}[thm]{Definition}

\begin{document}
\title{Stabilizing Heegaard splittings of high-distance knots}
\author{George Mossessian}
\maketitle
\begin{abstract}Suppose $K$ is a knot in $S^3$ with bridge number $n$ and bridge distance greater than $2n$. We show that there are at most ${2n\choose n}$ distinct minimal genus Heegaard splittings of $S^3\setminus\eta(K)$. These splittings can be divided into two families. Two splittings from the same family become equivalent after at most one stabilization. If $K$ has bridge distance at least $4n$, then two splittings from different families become equivalent only after $n-1$ stabilizations. Further, we construct representatives of the isotopy classes of the minimal tunnel systems for $K$ corresponding to these Heegaard surfaces. 
\end{abstract}
\section{Introduction}
In 1933, Reidemeister and Singer independently showed \cite{rei33, sin33} that any two Heegaard splittings of a 3-manifold become equivalent after some finite number of stabilizations to each. No upper bound on the genus of this common stabilization was known until 1996, when Scharlemann and Rubinstein \cite{scharl96} showed that, for a non-Haken manifold, there is an upper bound which is linear in the genera of the two respective Heegaard splittings. Later, they extended these results to a quadratic bound in the case of a Haken manifold. In 2011 Johnson \cite{joh11stab} showed that, for any 3-manifold there a linear upper bound $3p/2+2q-1$, where $p\geq q$ are the genera for the Heegaard surfaces. Examples of Heegaard splittings which required more than one stabilization to the larger-genus surface were not known until Hass, Thompson and Thurston \cite{hass03} constructed examples with stable genus $p+q$ in 2009. Later that year, Bachman \cite{bach09} gave examples with stable genus $p+q-1$, and Johnson \cite{joh10} gave a combinatorial description of the geometric construction give by Hass-Thompson-Thurston.

One goal of this paper is to show that for most $n$-bridge knots, there are many pairs of genus-$n$ Heegaard splittings of the knot exterior which have stable genus $2n-1$. This is accomplished by means of the same machinery that Rubinstein and Scharlemann developed and Johnson refined for their respective upper bounds, the Rubinstein-Scharlemann graphic. 

It was not previous known whether a knot complement can have Heegaard surfaces which require more than one stabilization to become equivalent. Theorem \ref{mainthm} shows that such knots do exist. Though we do not prove it here, it is easy to see that the knots we consider are all hyperbolic. The 3-manifolds with stable genus $p+q$ constructed by Hass, Thompson and Thurston in \cite{hass03} are negatively curved, but not hyperbolic. Their methods can be used to construct a hyperbolic 3-manifold with two Heegaard surfaces with stable genus $p+q-4$, so Theorem \ref{mainthm} extends this bound, as well. The geometry of high-distance knot complements and tunnel systems is a topic for future investigation.

The starting point of this work is a theorem of Tomova which states that for a knot with sufficiently high distance, the minimal bridge sphere is unique up to isotopy. This allows us to give a combinatorial description of all of the minimal genus Heegaard splittings of the complement of such a knot. In turn, we use this description to come up with a list of all the tunnels belonging to any minimal tunnel system for such a knot. It turns out that this description precisely generalizes Kobayashi's \cite{kob99} classification of unknotting tunnels for 2-bridge knots. 

This work suggests some questions for further investigation. In particular, is it possible to extend this construction to closed manifolds? And, can a criterion be found for determining whether two Heegaard splittings of a knot complement with low-genus common stabilization are isotopic, as Morimoto and Sakuma \cite{mor91} did for 2-bridge knots? 

In section 2 we will define all the terms necessary for stating the main theorem. In section 3, we construct these Heegaard splittings. It follows from work of Tomova \cite{tom07} that there are no others. In section 4, we show how they divide into two families, and we show that two splittings from the same family become equivalent after one stabilization. Section 5 is devoted to recalling the machinery of Johnson developed in \cite{joh08} and \cite{joh10}, and proving that it can be used in this situation. In section 6 this machinery is used to place a lower bound on the stable genus of two splittings from different families. Finally, in section 7, we show how to construct a tunnel system which corresponds to one of the Heegaard splittings constructed in the proof of Theorem \ref{mainthm}, so that any minimal tunnel system for $K$ is equivalent to one of these tunnel systems.

I would like to thank Abby Thompson for many discussions and suggestions without which this work would not have been possible, and Jesse Johnson for the initial conversation which led to this work.

\section{Preliminaries}

Let $M$ be a closed, compact, connected, orientable 3-manifold. A \emph{Heegaard surface} for $M$ is a closed, orientable, connected surface $\Sigma$ embedded in $M$ so that $\overline{M\setminus \Sigma}$ is a pair of handlebodies $H_1$, $H_2$. The triple $(\Sigma, H_1, H_2)$ is a \emph{Heegaard splitting} of $M$. The minimal genus of all such surfaces $\Sigma$ is the \emph{Heegaard genus} of $M$. Two Heegaard splittings $(\Sigma, H_1, H_2)$ and $(\Sigma', H_1', H_2')$ are \emph{isotopic} or \emph{equivalent} if there is an ambient isotopy of $M$ taking $\Sigma$ to $\Sigma'$ and $H_i$ to $H_i'$. Otherwise, they are \emph{distinct}. For a manifold with boundary, there is a natural generalization of this idea, discussed below.

Given a handlebody $H$, let $\Gamma \subset H$ be a graph which is a deformation retract of $H$. Then $\Gamma$ is a \emph{spine} of $H$. Removing a regular $N\subset \Int H$ neighborhood of some (possibly empty, possibly disconnected) subgraph of $\Gamma$ contained in $H$ results in a \emph{compression body}. However, we assume that no component of the boundary of this regular neighborhood is spherical. The spine of this compression body is the remainder of $\Gamma$, union the boundary of the regular neighborhood which was removed. If $M$ is a compact orientable manifold with boundary, then $\partial M$ is a closed orientable surface, and thus a Heegaard splitting of $M$ can be defined as a triple $(\Sigma, W_1, W_2)$ where $W_i$ are compression bodies. 

If $K\subset M$ is a knot, a \emph{bridge surface} for $K$ is a Heegaard surface $\Sigma$ for $M$ such that $K$ intersects $\Sigma$ transversely and $K\cap H^\pm$ is a boundary parallel collection of arcs in each handlebody, called \emph{bridge arcs}, or $K\cap\Sigma=\emptyset$ and $K$ is parallel into $\Sigma$. Equivalently, $K\cap H^\pm$ is a collection of arcs $\gamma_i$ such that there is a collection of arcs $(\alpha_i,\partial\alpha_i)\subset(\Sigma,\partial\gamma_i)$ which, together with the $\gamma_i$, cobound disks $D_i$ which, within each handlebody, are pairwise disjoint. These are called {\emph bridge disks}. If the genus of $\Sigma$ is $g$ and the number of points of $\Sigma\cap K$ is $2n$, then $\Sigma$ is a $(g,n)$ bridge surface. If $n$ is minimal over all surfaces of some fixed genus $g$, then $K$ is said to be a $(g,n)$-bridge knot. If $g=0$, this is reduced to simply saying that $K$ is $n$-bridge, and $\Sigma$ is a \emph{bridge sphere}. 

If a simple closed curve $\gamma$ on a Heegaard or bridge surface $\Sigma$ does not bound a disk in $\Sigma$ or a disk punctured once by $K\cap \Sigma$ in $\Sigma$, we say that $\gamma$ is \emph{essential}. We define the \emph{curve complex} $C(\Sigma)$, first introduced by Hempel in \cite{hem01}, of the surface $\Sigma$ as follows: let vertices correspond to isotopy classes of essential simple closed curves on $\Sigma$, and let any two vertices corresponding to curves which have disjoint representatives on $\Sigma$ be connected by an edge. To define the distance of the splitting $(\Sigma, H_1,H_2)$, let $K_1$ and $K_2$ be the collections of vertices corresponding to curves that bound disks to either side of $\Sigma$, which are disjoint from $K$ in the case of a bridge splitting. Then the minimal length edge path in $C(\Sigma)$ between a vertex of $K_1$ and a vertex of $K_2$ is called the \emph{distance} of $\Sigma$, and is denoted $d(\Sigma)$. If $\Sigma$ is a bridge sphere for a knot $K\subset S^3$ which realizes the minimal distance of all bridge spheres for $K$, then we say that the knot has distance $d(\Sigma)$, which is denoted as $d(K)$. 

Let $(\Sigma, W_1, W_2)$ be a genus-$g$ Heegaard splitting of $M$. If $\gamma$ is a properly embedded arc in, say, $W_2$, which is parallel into $\Sigma$, let $W'_1=W_1\cup N(\gamma)$, and $W'_2=\overline{W_2\setminus N(\gamma)}$, with $\Sigma'=\partial W'_2=\partial W'_1$. Then $W'_2$ and $W'_1$ are still compression bodies, and $(\Sigma', W'_1, W'_2)$ is a genus $g+1$ Heegaard splitting for $M$ which is a \emph{stabilization} of $(\Sigma, W_1, W_2)$. If $\Sigma$ was a $(g,n)$ bridge surface for a knot $K$, and $\gamma$ a bridge arc, then $\Sigma'$ is a \emph{meridional stabilization} of $\Sigma$ and is a $(g+1,n-1)$-bridge surface. Reidemeister \cite{rei33} and Singer \cite{sin33} showed independently that any two inequivalent Heegaard surfaces for $M$ have stabilizations which are equivalent. The minimal genus of this common stabilization is called the \emph{stable genus} of the two surfaces. 

Let $X_K=S^3\setminus\eta(K)$. In this paper, we classify the minimal genus Heegaard splittings and their common stabilizations for $X_K$ when $K$ is a high-distance knot:

\begin{thm}Let $K\subset S^3$ be an $n$-bridge knot with $n\geq 3$ and bridge distance $d>2n$. Then the Heegaard genus of $X_K$ is $n$, and there are at most ${2n\choose n}$ distinct minimal genus Heegaard splittings of $X_K$. If two Heegaard surfaces have $K$ on the same side, then they have a stable genus of at most $n+1$; If two Heegaard surfaces have $K$ on opposite sides, then their stable genus is at least $\min\left(2n-1,\dfrac{1}{2}d\right)$. \label{mainthm}
\end{thm}

\begin{cor}Let $K\subset S^3$ be an $n$-bridge knot with $n\geq 3$ and bridge distance $d>2n$. Then the tunnel number of $K$ is $n-1$, and there are at most ${2n\choose n}$ distinct tunnel systems for $K$, up to isotopy and edge slides.\label{cor-tunnels}\end{cor}

This extends Kobayashi's classification of unknotting tunnels for $2$-bridge knots in \cite{kob99}: the $\displaystyle{4\choose 2}=6$ unknotting tunnels classified there correspond to genus-2 Heegaard splittings of $S^3\setminus K$ all of which become equivalent after at most one stabilization. 

\section{Constructing Heegaard splittings}

The following is a special case of a theorem of Tomova: 
\begin{thm}[Theorem 10.3 in \cite{tom07}]
Let $K$ be a nontrivial knot in $S^3$ with bridge number $n\geq 3$. Let $S$ be a bridge sphere for $K$. If $\Sigma$ is a bridge surface for $K$, or a Heegaard surface for $X_K$, then either $\Sigma$ is obtained from $S$ by a sequence of meridional stabilizations and stabilizations, or $d(S)\leq 2-\chi(\Sigma-K)$.\label{maggy-10.3}
\end{thm}

Thus, if $K$ is $n$-bridge and $d(K)>2n$, then every non-stabilized $(g,n)$ bridge surface for $K$, and every genus-$n$ Heegaard surface for $X_K$, is constructed by meridionally stabilizing the unique bridge sphere $S$. This in turn implies that no sequence of meridional stabilizations of $S$ yields a surface which is stabilized. This also implies that the bridge sphere $S$ which realizes the bridge distance is unique up to isotopy. 

For the purposes of establishing some notation, we prove the following obvious lemma:

\begin{lem} If $\Sigma'$ is obtained from $(g,n)$-bridge surface $\Sigma$ by a single meridional stabilization,  then $\Sigma'$ is a $(g+1,n-1)$-bridge surface.\label{lem-merid-stab}\end{lem}
\begin{proof}
Note that, in order to be able to perform a meridional stabilization, $n$ must be greater than $0$.

Let $V,W$ be the handlebodies on either side of $\Sigma$ and $\gamma$ an arc of $K\cap V$. Let $W'=W\cup N(\gamma)$ and $V'=\overline{V\setminus N(\gamma)}$ so that $V'\cap W'=\Sigma'$. Since $\gamma$ can be isotoped into $\partial V$, there is a disk $D_\gamma\subset V$ so that $\partial D_\gamma=\gamma\cup\alpha$,  $\alpha\subset\partial V$. Let $D_\gamma'=\overline{D_\gamma\setminus N(\gamma)}$. If $\varphi:N(\gamma)\rightarrow I\times D^2$ is the obvious homeomorphism sending $\gamma$ to $I\times\{0,0\}$, let $D=\varphi^{-1}\left(\left\{\dfrac{1}{2}\right\}\times D^2\right)$ be a meridian disk of $N(\gamma)$. Since $D_\gamma'\cap N(\gamma)$ is an arc parallel to $\gamma$, $D_\gamma'\cap D$ is a single point, and therefore $\Sigma'$ is a stabilization of $\Sigma$, and is a Heegaard splitting of $S^3$. The genus of $\Sigma'$ is $g+1$, and the number of arcs of $K$ in $W'$ and $V'$ each is $n-1$.
\end{proof}

\begin{figure}[h!]
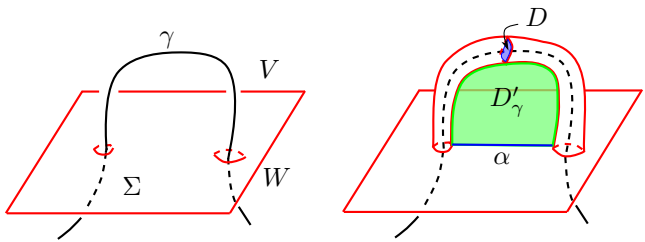
\caption{Diagram of proof of Lemma \ref{lem-merid-stab}.}
\label{fig-merid-stab}
\end{figure}

If $\Sigma$ is an $n$-bridge sphere, a sequence of $n$ meridional stabilizations results in an $(n,0)$-bridge surface $\Sigma'$. If the two handlebodies on either side of $\Sigma'$ are $V'$ and $W'$, one of them contains the knot $K$, say $V'$. $V'$ has $n$ compression disks, $D_i$ (the meridian disks $D$ of $N(\gamma)$ in the proof of Lemma \ref{lem-merid-stab}) which each intersect $K$ in exactly one point. Thus $K$ is parallel to a subset of a spine of $V'$, and therefore $\Sigma'$ is also a Heegaard surface for the knot complement $X_K$, with handlebody $W'$ and compression body $\overline{V'\setminus N(K)}$. 

\begin{dfn}
A \emph{tubed Heegaard surface} is a Heegaard surface for $X_K$ which is constructed by meridional stabilizations from a bridge sphere for $K$ in $S^3$.\end{dfn}

Note again that Theorem \ref{maggy-10.3} implies that, if $K$ is $n$-bridge and $d(K)>2n$, then every minimal genus Heegaard surface of $X_K$ is a tubed Heegaard surface.

The isotopy class of $\Sigma'$ may depend on the choice of meridional stabilizations, or the order in which they are performed. For example, Figure \ref{2bridge_splittings} shows three possible splittings of a 2-bridge knot complement resulting from meridional stabilizations of a bridge sphere. Though Theorem \ref{mainthm} concerns knots with bridge number $3$ or greater, Figure \ref{2bridge_splittings} presents heuristic pictures of $2$-bridge knots for simplicity. For a $3$-bridge knot with distance $6$ or greater, the constructions in figures \ref{2bridge-1} and \ref{2bridge-2} do not represent the same surface, because $K$ is on opposite sides of the resulting surface (see section 4). For a figure-8 knot, these two Heegaard splittings correspond to the two unique unknotting tunnels. The surface in Figure \ref{2bridge-3} may or may not be isotopic to the surface of Figure \ref{2bridge-1}. For the figure-8 knot, they are indeed isotopic, but the question of when they are distinct remains open for knots with bridge number $2$ or greater. For $2$-bridge knots, it is answered by Morimoto and Sakuma in \cite{mor91}.

\begin{figure}[h!]
\subfigure[]{\includegraphics[width=0.3\textwidth]{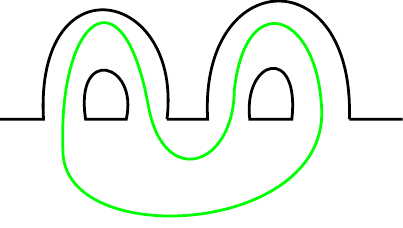}\label{2bridge-1}}
\subfigure[]{\includegraphics[width=0.3\textwidth]{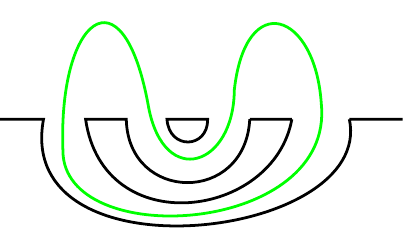}\label{2bridge-2}}
\subfigure[]{\includegraphics[width=0.3\textwidth]{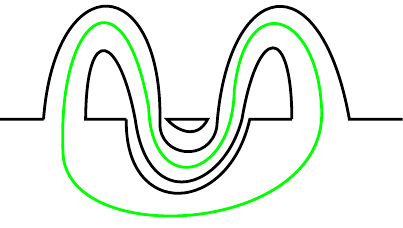}\label{2bridge-3}}
\caption{Three possible splittings of a 2-bridge knot complement.}
\label{2bridge_splittings}
\end{figure}

Before proving the next lemma, we note some definitions. If $S$ is a bridge sphere for $K$, a t ubed Heegaard surface $\Sigma$ may be viewed as $\left(S\setminus \cup_{x\in S\cap K}N(x)\right)\bigcup\left(\cup_{i=1}^n A_i\right)$, where the neighborhoods $N(x)$ are pairwise disjoint, and the $A_i$ are annuli with $A_i\cap S=\partial A_i$ $\forall i$. Each such annulus $A_i$ is the image of an embedding \[\varphi_i:(S^1\times I,S^1\times\{0\},S^1\times\{1\})\rightarrow(S^3\setminus K,\partial N(x_i),\partial N(x_j)),\] and the images of all the $\varphi_i$ are pairwise disjoint.

\begin{dfn}Let $\gamma, \alpha, D_\gamma, D_\gamma'$ be as in Lemma \ref{lem-merid-stab}, let $z\in S^1$ and choose $\varphi_i$ so that $\varphi_i(\{z\}\times I)=D_\gamma\cap A_i$. $D_\gamma\setminus D_\gamma'$ is a square with sides $\gamma$, $\partial D_\gamma'\setminus\alpha$, and two components of $\alpha\cap N(\gamma)$. Given an orientation of $K$ and orienting $I$ from 0 to 1, the annulus $A_i$ \emph{runs along $K$} if the arcs $\gamma$ and $\varphi\left(\{z\}\times I\right)$ with the orientations induced by $I$ are parallel in the square.\label{dfn-runs-along}\end{dfn}

The point of Definition \ref{dfn-runs-along} is just to have a unique way of identifying the annuli by one of the endpoints, as follows:

\begin{dfn}If $A_i$ is an annulus which runs along $K$ with $\varphi_i(S^1\times\{0\})\cap S=\partial N(x_i)$ and $\varphi_i(S^1\times\{1\})\cap S=\partial N(x_j)$, then $A_i$ is said to have its \emph{left foot at $x_i$} and \emph{right foot at $x_j$}. The difference $j-i\mod 2n$ is the \emph{length} of the annulus.\label{dfn-feet}\end{dfn}

We take the convention that the annulus $A_i$ has left foot at $x_i$, and $r(i)$ denotes the index of the right foot of the annulus.

\begin{dfn}With $K$ an $n$-bridge knot and $S$ a bridge sphere, let $S\cap K=\{x_1,...,x_{2n}\}$, labelled by picking a point on $K$, following $K$ along a fixed orientation, and labelling the points of $S\cap K$ in order of increasing index. Let $\Sigma$ be a tubed Heegaard surface, so that $\Sigma=\left(\overline{S\setminus\cup_{i=1}^{2n}N(x_i)}\right) \cup\left(\cup_{i=1}^nA_i\right)$ so that each $A_i$ runs along $K$. Let $I_0=\{\alpha_i\}_{i=1}^n$ be the set of indices from $1$ to $2n$ that are left feet of the annuli, that is, $\varphi_i(S^1\times\{0\})\cap S=\partial N(x_{\alpha_i})$. The set $I_0$ is the \emph{index} of the tubed Heegaard surface $\Sigma$ for the fixed orientation of $K$ and labelling of $S\cap K$. See Figure \ref{fig-index}.\label{dfn-index}\end{dfn} 

\begin{figure}[h!]
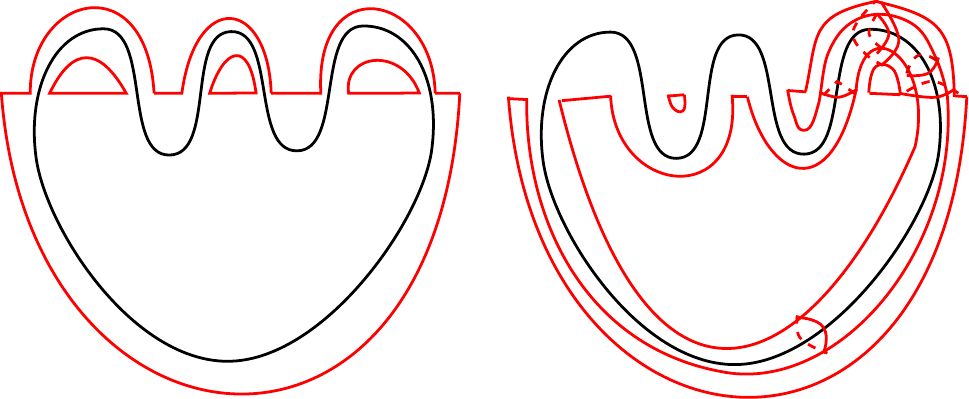
\caption{}
\label{fig-index}
\end{figure}

Since every tubed Heegaard surface has an index, this definition gives the upper bound of $\displaystyle{2n \choose n}$ on the number of tubed Heegaard surfaces of $K$. Now we show that every choice of $n$ numbers is the index of some tubed Heegaard surface:

\begin{lem}Let $I_0$ be any unordered list of $n$ elements of $\{1,2,...,2n\}$ without repititions. This $I_0$ can be realized as the index of a tubed Heegaard surface.\label{lem-2n-choose-n}\end{lem}
\begin{proof}
The desired surface is constructed by attaching the annuli to $S\setminus\cup_iN(x_i)$ in order of increasing length, which is equivalent to performing meridional stabilizations in the correct order. Define $J_0=\emptyset$. For each $1\leq k\leq n$, let 
\[J_k=\left\{i_j\in I_0\setminus\left(\cup_{\ell=0}^{k-1}J_\ell\right)\Big| i_j+(2k-1)\mod 2n\not\in I_{k-1}\right\}\] and 
\[I_k=I_{k-1}\cup\left\{i_j+(2k-1)\mod 2n\Big| i_j\in J_k\right\}.\] 
For each $i_j\in J_k$, attach an annulus with left foot at $x_{i_j}$ and right foot at $x_{i_j+2k-1}$, where the indices are all computed $\mod 2n$.

Thus, at each $k$, $I_k$ is the set $I_0$ together with the indices of points which are right feet of annuli that have already been attached, that is, indices which are no longer available for attaching annuli, and $J_k$ is the set of indices of left feet of annuli of length $2k-1$. After gluing the annuli of length $2n-1$ with left feet in $J_n$, the resulting tubed Heegaard surface has index $I_0$ by construction. Thus every such set is the index of some tubed Heegaard surface, proving the lemma.
\end{proof}

Lemma \ref{lem-2n-choose-n} together with the uniqueness of the $n$-bridge sphere given by Theorem \ref{maggy-10.3} imply the following corollary: 
\begin{cor}
 If $K$ is an $n$-bridge knot with bridge distance at least $2n$, there are at most ${2n\choose n}$ distinct minimal Heegaard splittings of $X_K$.
\end{cor}

\section{Common stabilizations of same-side tubed Heegaard splittings}

The set of all tubed Heegaard surfaces for $K$ can be divided into two families, depending on which side of the initial bridge sphere is adjacent to $K$ in the final splitting. In this section, we prove that if two tubed Heegaard surfaces lie on the same side of $K$, they become equivalent after just one stabilization. 

\begin{dfn}
Let $\Sigma\subset S^3$ be a separating surface, and $\A$ an annulus properly embedded in $S^3\setminus \Sigma$. Let $\Sigma'$ be obtained from $\Sigma$ by cutting $\Sigma$ along $\partial \A$ and gluing in two copies of $\A$. We say that $\Sigma'$ is obtained from $\Sigma$ by an \emph{annulus compression} along the \emph{compressing annulus} $\A$. See Figure \ref{fig-ann-comp}.
\end{dfn}

\begin{dfn}
Let $\A$ be a compressing annulus for a tubed Heegaard surface such that $\partial\A$ consists of essential loops contained in two annuli of $\Sigma$, $\gamma_i\subset A_i$ and $\gamma_j\subset A_j$. Let $\psi:S^1\times I\rightarrow\A\hookrightarrow S^3$ be the homeomorphism with image $\A$, so that $\psi(S^1\times\{0\})=\gamma_i$ and $\psi(S^1\times\{1\})=\gamma_j$. With the orientation on $I$ from $0$ to $1$, if the induced orientation on $\psi(\{z\}\times I)$ for some $z\in S^1$ is parallel to the fixed orientation of $K$, then $\A$ \emph{runs along} $K$ and is denoted as $[ij]$. Further, to allow for the case $i=j$ without ambiguity, we require that $[ij]$ have $\gamma_i=\partial N(x_{r(i)})$ and $\gamma_j=\partial N(x_j)$. Note that the compressing annulus $[ij]$ is therefore different from the compressing annulus $[ji]$. \label{dfn-compressing-annulus-runs-along}
\end{dfn}

\begin{lem}If $\Sigma$ is a tubed Heegaard surface for $K$ and $\A$ a compressing annulus $[ij]$ with $i\neq j$, then an annulus compression along $\A$ is equivalent to a single stabilization and destabilization of $\Sigma$. If the index of $\Sigma$ is $I_0$, then after such an annulus compression the index of the resulting tubed Heegaard surface is $\sigma(I_0)$, where $\sigma\in S_{2n}$ is the 2-cycle $(r(i)\ \ j)$.
\label{lem-annulus-stabilization}\end{lem}
\begin{proof}
Let $\psi$ be as in Definition \ref{dfn-compressing-annulus-runs-along}, and $\beta=\psi\left(\{z\}\times I\right)$ be a spanning arc in $\A$ connecting $\gamma_i$ and $\gamma_j$. Let $N(\beta)$ be a small enough neighborhood of $\beta$ in $S^3$ so that $N(\beta)\cap \Sigma$ is just two disks in $\Sigma$ which are neighborhoods of $\partial\beta$, $D_i$ and $D_j$. Since $\A$ runs along $K$, $\beta$ is parallel into $\Sigma$, and therefore there is a disk $\Delta$ properly embedded in $S^3\setminus \Sigma\setminus\A$ which intersects a meridian disk of $N(\beta)$ in exactly one point, so $\Sigma' = (\Sigma-D_i-D_j)\cup\partial N(\beta)$ is a stabilization of $\Sigma$. 

Consider $D=\A\setminus N(\beta)$. $\partial D$ consists of $\gamma_i\setminus N(\beta)$, $\gamma_j\setminus N(\beta)$, and two spanning arcs on $N(\beta)$. Since both $A_i$ and $A_j$ are the results of meridional stabilization, they both have the disks $D_\gamma'$ described in the proof of Lemma \ref{lem-merid-stab}, and since $i\neq j$ either of these disks only intersect $D$ in one point. Therefore compressing along $D$ is a destabilization of $\Sigma'$, and the end result is identical to performing an annulus compression along $\A$. 

The annulus compression $[ij]$ has the effect of connecting the left foot of $A_i$ to the right foot of $A_j$, and the right foot of $A_i$ is now the left foot of an annulus that ends at the left foot of $A_j$. Therefore the index of the new tubed Heegaard surface is $I_0$ with $r(i)$ instead of $j$. See Figure \ref{fig-ann-comp}
\end{proof}

\begin{figure}
\def\svgwidth{\linewidth}
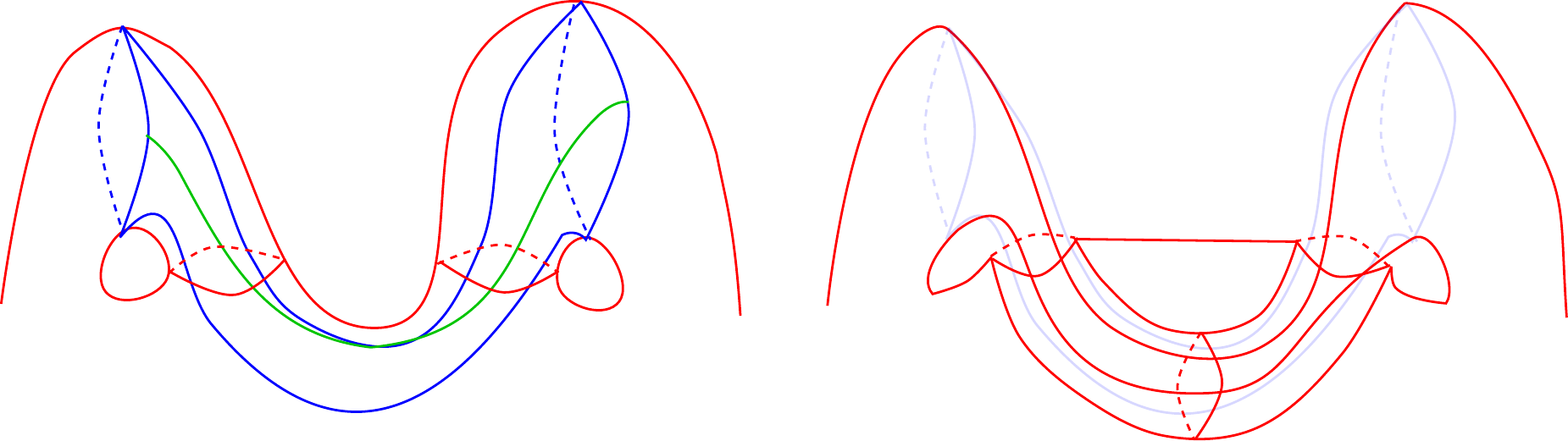
\caption{}
\label{fig-ann-comp}
\end{figure}

Observe that if $i=j$, the surface after compressing along $[ij]$ is disconnected, with one boundary-parallel component. For this reason, we will always assume that $i\neq j$ unless stated otherwise.

\begin{dfn}
Given a bridge sphere for $K$, label one side $+$ and the other $-$. Let $\Sigma$ be a tubed Heegaard surface derived from this bridge sphere. If $K$ is adjacent to the $-$ side of the portion of the bridge sphere outside the $N(x_i)$, $K$ is \emph{below} $\Sigma$. Else, $K$ is \emph{above} $\Sigma$. \label{dfn-above-below}
\end{dfn}

%%%%%%%%%%%%%%%%%%%%%%%%%%%%%%%%%%%% Page 6 Dfn 3.4

It is also possible to think of above and below in terms of Morse functions: let $K\subset S^3$, and $f:S^3\rightarrow [-1,1]$ be a Morse function on $S^3$ so that $f^{-1}(t)$ is a 2-sphere for $t\in(-1,1)$, $f|_K$ is also Morse with image $\left[-\dfrac{1}{2},\dfrac{1}{2}\right]$, and $S=f^{-1}(0)$ is a bridge sphere for $K$. Let $\Sigma$ be a tubed Heegaard surface for $K$ from $S$ such that $f|_\Sigma$ has image $\left[-\dfrac{3}{4},\dfrac{3}{4}\right]$. If $f^{-1}(-1)$ is a point in the handlebody containing $K$, we say that $K$ is \emph{below} $\Sigma$. If, on the contrary, $f^{-1}(1)$ is contained in the handlebody containing $K$, then $K$ is \emph{above} $\Sigma$.

There are two obvious tubed Heegaard surfaces for any knot in bridge position: those with index $(1,3,...,2n-1)$ and $(2,4,...,2n)$. We take the convention that in the first, $K$ is below $\Sigma$, and in the second, $K$ is above $\Sigma$. Thus these will be referred to as $\Sigma_b$ and $\Sigma_a$ respectively.

If there is a genus-$g$ tubed Heegaard surface $\Sigma_0$ with index $I_0$ such that $i,j\in I_0$, compressing along the annulus $[i_1j_1]$ yields another splitting $\Sigma_1$ with index $I_1$. Considering the annulus compression as a stabilization followed by a destabilization as in the proof of Lemma \ref{lem-annulus-stabilization}, call the intermediate surface with genus $g+1$ $\Sigma_0'$. Another annulus compression along $[i_2j_2]$ so that $i_2,j_2\in I_1$ is equivalent to another stabilization-destabilization pair resulting in a surface $\Sigma_2$, with intermediate genus-$g+1$ surface $\Sigma_1'$. However, stabilizations are unique, so $\Sigma_0'$ is in fact equivalent as a Heegaard surface to $\Sigma_1'$, and so, after a sequence of handle slides followed there is a single destabilization of $\Sigma_0'$ which results in the surface $\Sigma_2$ directly. Therefore, extending this argument inductively, we have proved:

\begin{lem}Any two genus-$g$ tubed Heegaard surfaces related by a sequence of $i\neq j$ annulus compressions have a common stabilization of genus $g+1$.\label{lem-g+1}\end{lem}

We also observe the following: 
\begin{lem}Two tubed Heegaard surfaces have $K$ on the same side if and only if they are related by a sequence of annulus compressions of the $[ij]$ type.\end{lem}
\begin{proof}
First, we show that an annulus compression does not change which side of the surface the knot lies on. Pick a point on $K$, and consider a path $\gamma$ transverse to $\Sigma$ from this point to $f^{-1}(-1)$, where $f$ is the fixed Morse function defined above. Let $\nu=|\gamma\cap\Sigma|$. $\nu=0\mod 2$ if and only if $K$ is below $\Sigma$. Let $\Sigma'$ be the resullt of performing the annulus compression along $[ij]$ on $\Sigma$, and $\nu'=|\gamma\cap\Sigma'|$. If $|\gamma\cap[ij]|=m$, then $\nu'=\nu+2m$, and so $K$ is on the same side of $\Sigma'$ as of $\Sigma$.

Conversely, suppose without loss of generality that $\Sigma_1$ and $\Sigma_2$ are two tubed Heegaard surfaces with $K$ above both. If $S$ is the unique bridge sphere that $\Sigma_i$ are constructed from, and $\gamma_1,...,\gamma_{2n}$ consecutive bridge arcs of $K$ with respect to $S$, we let $p_i$ (resp. $q_i$) be the number of annuli running over $\gamma_i$ in $\Sigma_1$ (resp. $\Sigma_2$). Since $K$ is above both $\Sigma_1$ and $\Sigma_2$, $p_i-q_i=(0\mod 2)$ for all $i$. Also note that $p_{i+1}=p_i\pm 1$, and similarly for $q_i$, so $\Delta=(p_{i+1}-q_{i+1})-(p_i-q_i)$ is either $0$ or $\pm 2$. Therefore, we can construct a collection of compressing annuli that will transform one surface into the other: if $\Delta=0$, then the number of compressing annuli over $\gamma_i$ is equal to the number passing over $\gamma_{i+1}$. Otherwise, a compressing annulus has either a left or right foot at $\partial N(\gamma_i\cap\gamma_{i+1})$ for $\Delta=-2$ or $2$ respectively. This gives a way to construct a sequence of compressing annuli which takes $\Sigma_1$ to $\Sigma_2$. 

\end{proof}

Thus we have proved the main result of this section:
\begin{lem}
Given two minimal genus tubed Heegaard surfaces for $X_K$ where $K$ is an $n$-bridge knot with bridge distance at least $2n$, if $K$ lies above (resp. below) both surfaces, then they become equivalent after one stabilization.
\end{lem}

%%%%%%%%%%%%%%%%%%%%%%%%%%%%%%%%%%%%%%%%%%%%%%%%%%%%%%%%%%%%%%%%%
%%%%%%%%%%%%%%%%%%%%%%%%%%%%%%%%%%%%%%%%%%%%%%%%%%%%%%%%%%%%%%%%%

\section{Sweep-outs and the graphic}

In this section, we recall the machinery developed in \cite{joh08} and used in \cite{joh10} to construct Heegaard splittings with large stable genus, and modify it for application in this situation.%After establishing some preliminaries, in sections 4.2 and 4.3 we show that there are two situations which give bounds on the Euler characteristic of one splitting surface with respect to either the Euler characteristic or distance of another splitting surface. In section 4.4, we prove a lemma that will allow us to place a common stabilization of two opposite-sided tubed Heegaard surfaces in one of these two situations, and in section 4.5 we show that any surface isotopic to a common stabilization must be in one of the two aforementioned situations, allowing us to bound its Euler characteristic from above. 

A $(g,n)$-bridge surface $S$, with $n>0$, in $(S^3,K)$ splits the 3-sphere into two handlebodies $H^\pm$, each containing a boundary-parallel collection of $n$ properly embedded arcs, $\tau^\pm$. If we restrict $S$ to $X_K$, it splits $X_K$ into two handlebodies with neighborhoods of the $n$ arcs removed (we note that this is homeomorphic to a genus-$g+n$ handlebody). To begin generalizing this machinery to punctured handlebodies, we define a \emph{spine} for $H^+\setminus\eta(\tau^+)$ (or $H^-\setminus\eta(\tau^-)$) to be a graph embedded in $H^+\setminus\eta(\tau^+)$ which contains the spine of $H^+$, together with an edge along a meridian circle on the boundary of each arc of bridge, and a vertical edge from a vertex of the spine of $H^+$ to each of the meridian circles. Note that a regular neighborhood of the spine on either side of $S$ in $X_K$ has boundary which is equivalent to $S-\eta(K)$, and that $H^\pm\setminus\eta(\tau^\pm)$ deformation retracts onto its spine. In the case that $g=0$ and $H^\pm$ are 3-balls, the spine of $H^+$ should be chosen to be a single vertex. We will call this the \emph{central vertex}.

Let $\partial_-W$ be some collection of closed positive-genus surfaces. The compression body $W$ is formed by attaching 1-handles to $\partial_-W\times I$. Then a \emph{spine of $W$} is a graph with valence-1 vertices on the components of $\partial_-W$ union $\partial_-W$, such that $W$ deformation retracts onto this union of graph and surfaces. If $W$ is a handlebody, then $\partial_-W=\emptyset$ and the spine is just a graph which $W$ is a regular neighborhood of. For a Heegaard splitting $(\Sigma, W_1, W_2)$ of some manifold $M$, where either $W_i$ could be a compression body, $\Sigma$ is equivalent to the positive boundary $\partial_+W_i$ of a regular neighborhood in $X_K$ of either the spine of $W_1$ or $W_2$. 

A \emph{splitting surface} in $X_K$ is either a bridge surface for $(S^3,K)$ restricted to $X_K$, or a Heegaard surface for $X_K$. A \emph{sweep-out} for $X_K$ is a map $f:S^3\setminus\eta(K)\rightarrow[-1,1]$ such that for each $t\in(-1,1)$, $f^{-1}(t)$ is equivalent to some splitting surface $\Sigma$, and $f^{-1}(\pm1)$ is the two spines to either side of $\Sigma$. We say that $f$ is a \emph{sweep-out given by $\Sigma$}.

Let $f$ be a sweep-out of $X_K$ given by a tubed Heegaard surface $\Sigma$, and $g$ a sweep-out of $X_K$ given by the bridge sphere $S$. Given the map $f\times g:X_K\rightarrow[-1,1]^2=[-1,1]\times[-1,1]$, the \emph{graphic} of $f\times g$ is the set of points $(s,t)\subset[-1,1]^2$ for which $f^{-1}(s)=\Sigma_s$ is not transverse to $g^{-1}(t)=S_t$.   For a generic $f\times g$, the graphic $\Gamma$ is a $2$- and $4$-valent graph subset of $[-1,1]^2$ (with for one exception, noted below), with a finite number of points along any vertical or horizontal line $\{s\}\times[-1,1]$ or $[-1,1]\times\{t\}$, and no two vertices sharing an $s$ or $t$ coordinate. See \cite{joh08} for details. We note that at some subset of the four corners of the square, the spine corresponding to $\Sigma$ is not transverse to the spine corresponding to $S$ on the boundary of $X_K$. Generically, the closure of these points does not intersect the interior of the square, so their existence will not affect our arguments, so we ignore these points. 

Let $\Theta_s^{-}=f^{-1}([-1,s])$ and $\Theta_s^+=f^{-1}([s,1])$. We say that $S_t$ is \emph{mostly above} $\Sigma_s$ if each component of $S_t\cap \Theta_s^-$ is contained in a disk subset of $S_t$, or in an annulus of $S_t$ which has one boundary component in $\partial X_K$. Similarly $S_t$ is \emph{mostly below} $\Sigma_s$ if each component of $S_t\cap\Theta_s^+$ is contained in a disk subset of $S_t$ or in an annulus of $S_t$ which has one boundary component in $\partial X_K$. If there exist $t_0, t_1$ and $s$ such that $S_{t_0}$ is mostly below $\Sigma_s$ and $S_{t_1}$ is mostly above $\Sigma_s$, we say that $\Sigma$ \emph{spans} $S$. If $t_0<t_1$, then $\Sigma$ spans $S$ \emph{positively}, and if $t_1<t_0$, $\Sigma$ spans $S$ \emph{negatively}. Note that it is possible for $\Sigma$ to span $S$ both positively and negatively.

If there is some value of $s$ so that $\{s\}\times [-1,1]$ contains no vertices of $\Gamma$ and no points where $S_t$ is either mostly above or mostly below $\Sigma_s$, for all values $t\in[-1,1]$, then we say that $\Sigma$ \emph{splits} $S$. Note that in this case, for this value of $t$, every $S_t\cap \Sigma_s$ contains at least one loop that neither bounds a disk in $S_t$ nor is parallel to $\partial S_t$. A curve with these properties is called \emph{essential} in $S_t$.

If an essential curve on $\Sigma$ bounds an embedded disk $D$ in $X_K\setminus\Sigma$, this disk is a \emph{compressing disk} for $\Sigma$. If the curve bounds an once-punctured disk whose other boundary component is a meridian of $\partial X_K$, this once-punctured disk is a \emph{cut disk}. A \emph{compression} of a connected surface $\Sigma$ along the compressing disk $D$ is a surface equivalent to the set of boundary components of $N(\Sigma\cup D)$ which are not isotopic to $\Sigma$. A \emph{cut compression} of a connected surface $\Sigma$ along the cut disk $D$ is the set of boundary components of $N(\Sigma\cup D)$ which meet $N(D)$. A \emph{c-compression} along a \emph{c-disk} is either a compression along a compressing disk or a cut compression along a cut disk. 

The following lemma is a version of \cite[Theorem 3.1]{joh10}, and follows the same line of proof. We write it out in full to show how it adapts to the current situation.

\begin{lem} Let $\Sigma$ and $S$ be two splitting surfaces for $X_K$. Suppose $\Sigma$ spans $S$ both positively and negatively. Then there is a sequence of compressions and cut compressions of $\Sigma$ after which there are two components of the compressed surface which are parallel to $S$. \label{lem-span-pos-neg}\end{lem}
\begin{proof}
Let $t_-, t_0, t_+$ and $s$ be the values such that $S_{t_-}$ and $S_{t_+}$ are mostly below $\Sigma_s$, and $S_{t_0}$ is mostly above $\Sigma_s$. Every curve of $S_{t_-}\cap \Sigma_s$ is inessential in $S_{t_-}$. First, we claim that if, for some $t$, an innermost curve $\gamma\in S_{t}\cap \Sigma_s$ is also inessential in $\Sigma_s$, then $\gamma$ can be removed by isotopy of the surfaces.
 
\emph{Case 1: }If $\gamma$ is contained in a disk subset of $S_t\cap\Sigma_s$, then by the Jordan Curve Theorem $\gamma$ must itself bound a disk in $S_t$. If $\gamma$ also bounds a disk in $\Sigma_s$, then since $X_K$ is irreducible, the two disks cobound a ball across which $\Sigma_s$ can be isotoped. Note that in this case, $\gamma$ cannot bound an annulus in $\Sigma_s$, since this annulus together with the disk would form a sphere in $S^3$ punctured exactly once by $K$, but since $S^3$ contains no copies of $S_1\times S_2$ this is a contradiction.

\emph{Case 2: }If $\gamma$ is contained in an annulus subset of $S_t$, it bounds either a disk or an annulus in $S_t$. If it bounds a disk, we revert to Case 1. If it bounds an annulus, by the same argument as in Case 1 we note that $\gamma$ must also bound an annulus in $\Sigma_s$. If we replace $N(K)$ (i.e. perform a trivial Dehn filling) $\gamma$ bounds two disks in $S^3$, so the curve of intersection in $X_K$ can be removed by sliding the two boundary components of the annuli in $\partial X_K$ past each other. 

After removing all curves of intersection which are inessential in both surfaces, we consider a curve $\gamma\in S_{t_-}\cap \Sigma_s$ which is innermost in $S_{t_-}$ and essential in $\Sigma_s$. Therefore this curve bounds a $c$-disk for $\Sigma_s$, along which we compress $\Sigma_s$. We repeat this process with the compressed surface until there are no more curves of intersection of $S_{t_-}$ and the resulting compressed, possibly disconnected surface, which we will call $\Sigma_s'$. Note that now, $S_{t_-}$ lies entirely to one side of $\Sigma_s'$. 

We repeat the same process with all curves of intersection between $\Sigma_s'$ ad $S_{t_+}$, and call the resulting compressed surface $\Sigma_s''$, and observe that now $S_{t_-}$ and $S_{t_+}$ are both on the same side of $\Sigma_s''$, as they were both mostly below $\Sigma_s$. Finally, we repeat the process with all curves of intersection between $S_{t_0}$ and $\Sigma_s''$, resulting in a surface $\Sigma_s'''$ now disjoint from all three surfaces $S_{t_-}$, $S_{t_+}$ and $S_{t_0}$, with $S_{t_-}$ and $S_{t_+}$ to one side of $\Sigma_s'''$, and $S_{t_0}$ on the other side of $\Sigma_s'''$. 

Now we maximally compress and $\partial$-compress $\Sigma_s'''$ in the complement of $S_{t_-}\cup S_{t_0}\cup S_{t_1}$, and call the resulting surface $F$. $F$ is an incompressible surface properly embedded in a manifold homeomorphic to $S\times [0,2]$, a $2n$-punctured sphere cross interval. $F$ separates $S\times\{0\}$ from $S\times\{1\}$, and separates $S\times\{1\}$ from $S\times\{2\}$. Therefore $F$ has at least two incompressible, $\partial$-incompressible components, each inside a copy of $S\times I$ separating $S\times\{0\}$ from $S\times\{1\}$. We now show that each such component must be parallel to $S$. 

Let $F'$ such a component of $F$ in $S\times I$. If there is a component of $\partial S\times I$ which does not contain a loop of $\partial F'$ which separates $\partial S\times\{0\}$ from $\partial S\times\{1\}$, then $F'$ does not separate $S\times\{0\}$ from $S\times\{1\}$, so $\partial F'$ must meet each component of $\partial S\times I$ in a meridian circle of that component. If some component of $\partial S\times I$ meets $\partial F'$ in more than one meridian component, then by an innermost-arc argument there is either a compression or a $\partial$-compression of $F'$, a contradiction, so $F'$ meets each component of $\partial S\times I$ in exactly one meridian loop. Therefore we may glue in $2n$ copies of $D^2\times I$ into $S\times I$, and extend $F'$ by gluing in disks along those meridian loops of intersection. Then this surgered $F'$ becomes a separating, incompressible surface in $S^2\times I$, so the surgered $F'$ must be homeomorphic to $S^2$. 

Therefore $F'$ is homeomorphic to a $2n$-punctured sphere and was parallel to a copy of $S$, so $F$ contains at least two components which are copies of $S$.
\end{proof}

\begin{cor}
Let $K$ be an $n$-bridge knot. If $\Sigma$ is a Heegaard splitting for $X_K$ which spans a splitting sphere $S$ both positively and negatively, then the genus of $\Sigma$ is at least $2n-1$.\label{cor-2n-1}
\end{cor}
\begin{proof}
By Lemma \ref{lem-span-pos-neg}, there is a sequence of compressions and cut compressions of $\Sigma$ after which the resulting collection of surfaces contains two copies of $S$. Since $\chi(\Sigma)$ is non-decreasing under compressions and cut compressions, $\chi(\Sigma)\leq 2\chi(S)\Rightarrow 2-2g\leq 4-4n\Rightarrow g\geq 2n-1$
\end{proof}

%\subsection{Splitting sweep-outs}

The following lemma is a version of \cite[Theorem 4.2]{joh10}. The proof follows identical arguments as in \cite{joh08} and \cite{joh10}, so we only sketch details here, and direct the reader there for details. It is generalized to splitting surfaces the same way as Lemma \ref{lem-span-pos-neg} above. 

\begin{lem}If $S$ and $\Sigma$ are splitting surfaces for $X_K$, $\chi(\Sigma)<0$, and $\Sigma$ splits $S$, then $d(S)\leq 2-\chi(\Sigma)$.\label{lem-split}
\end{lem}

\begin{proof}[Sketch of Proof.]
Let $s$ be such that $\{s\}\times[-1,1]$ is disjoint from any vertices of $\Gamma$ and $S_t$ is neither mostly above nor below $\Sigma_s$ for any $t\in(-1,1)$. For every $t$, $\Sigma_s\cap S_t$ contains at least one curve essential in $S_t$. Let $g$ and $f$ be the sweep-outs given by $S$ and $\Sigma$ respectively. Let $[\alpha,\beta]$ be the largest interval such that for no $t\in(\alpha,\beta)$ is any curve of $\Sigma_s\cap S_t$ both essential in $S_t$ and bounds a disk or cut disk in either compression body to either side of $\Sigma_s$. (If no such essential interval exists, $d(S)\leq 1$ and we are done).

Consider $\Sigma'=\Sigma\cap g^{-1}([\alpha+\epsilon,\beta-\epsilon])$ for a small $\epsilon>0$. The projection $\pi:\Sigma'\rightarrow S$ which factors through inclusion into $S\times I$ will send isotopy classes of curves in $\Sigma'$ to isotopy classes of curves in $S$, and will send curves bounding pairs of pants in $\Sigma'$ to disjoint curves in $S$. Therefore the induced map $\pi^\ast:\Sigma'\rightarrow \C(S)\cup\{0\}$ projects a pants decomposition of $\Sigma'$ to a path in $\C(S)$, which has length at most $-\chi(\Sigma')$. Since $[\alpha,\beta]$ was chosen to be as large as possible, one more edge at each end of this path connects it to curves which bound disks in the handlebodies to either side of $\Sigma$, and $d(S)\leq 2-\chi(\Sigma')\leq 2-\chi(\Sigma)$.
\end{proof}

Further, we note that the stabilization of $\Sigma$ can be done in such a way as to preserve the spanning properties of the graphic.
\begin{lem}Let $\Sigma$, $S$ be splitting surfaces so that $\Sigma$ spans $S$ positively (resp. negatively), then a stabilization of $\Sigma$ also spans $S$ positively (resp. negatively).\label{lem-stab}\end{lem}
\begin{proof}
Let $f$, $g$ be sweep-outs given by $\Sigma$ and $S$ respectively, so that $\Sigma$ spans $S$ positively. Let $t_1, t_2, s$ be the values so that $S_{t_1}$ is mostly below $\Sigma_s$ and $S_{t_2}$ is mostly above $\Sigma_s$. Let $\Sigma'$ be a stabilization of $\Sigma_s$, and let $B$ be a ball which contains the cancelling pair of handles added to $\Sigma_s$ and intersects $\Sigma_s$ in a disk, contained in $g^{-1}((t_1,t_2))$. Let $f'$ be a sweep-out of $\Sigma'$ so that $f'^{-1}(s)\setminus B$ is setwise equal to $f^{-1}(s)\setminus B$. Since $S_{t_1}$ is mostly below $\Sigma_s$, and disjoint from $B$, it is also mostly below $\Sigma'_s$, and likewise $S_{t_2}$ is mostly above $\Sigma'_s$. Therefore $\Sigma'$ spans $S$ with the same sign that $\Sigma$ does.
\end{proof}

%\subsection{Isotopies of the graphic}  
Finally, the following lemma is proved by Johnson in \cite[Lemma 9]{joh11} for graphics of Morse functions. We reproduce the proof for sweep-outs with some minor modifications and clarifications.
\begin{lem}
Let $f_0$, $f_1$ be two sweep-outs given by equivalent Heegaard surfaces $\Sigma$, $\Sigma'$, and $g$ be a sweep-out given by the splitting surface $S$, so that $f_0\times g$ and $f_1\times g$ are both generic. Then there is a family of sweep-outs $f_r$, $r\in[0,1]$, so that for all but finitely many $r$, $f_r\times g$ is generic. At the non-generic values, the graphic has at most two valence-two or -four vertices at the same level, or one valence-six vertex.\label{lem-isotopy}
\end{lem}
The idea of the proof is as follows -- we approximate $\{f_r\}$ with a piecewise linear path in $C^\infty(X_K,\R)$, and consider consecutive vertices of this path, $\varphi_0$ and $\varphi_1$. Then we consider the image in $[-1,1]^3=\{(s_1,s_2,t)\}$ of points in $M$ where the surfaces given by $\varphi_0^{-1}(s_1)$, $\varphi_1^{-1}(s_2)$ and $g^{-1}(t)$ are not transverse, a ``3D graphic'' of sorts. Then we examine the projection of this image onto a plane in the cube which rotates about the $t$-axis from the $\varphi_0\times g$ plane to the $\varphi_1\times g$ plane, and classify the types of degeneracies which can occur. It can probably be proved for splitting surfaces in general with some attention given to singularites on the boundary, but we only need it for $\Sigma$, $\Sigma'$ as Heegaard surfaces. 
\begin{proof}
If $\Sigma$ and $\Sigma'$ are equivalent, there is an ambient isotopy of $M$, say $\phi_r$, which takes $\Sigma$ to $\Sigma'$ and takes the spines of $\Sigma$ to the spines of $\Sigma'$. Let $f_r=f\circ\phi_r$ be the sweep-outs given by the surfaces of this one-parameter family. Because Morse functions are dense in $C^\infty(M,\R)$, each $f_r$ can be approximated arbitrarily closely by a Morse function $f'_r$ in the $C^\infty$ topology, meaning that for any $\epsilon>0$, there exists a 1-parameter family of Morse functions $f'_r$ so that for $t\in[-1+\epsilon, 1-\epsilon]$, $f'^{-1}_r(t)=f_r^{-1}(t)$, so the graphic for $f'_r\times g$ is identical to that of $f_r\times g$ except on an arbitrarily small neighborhood of $\{-1,1\}\times I$. In addition, we approximate $g$ with a Morse function $g'$ so that the graphic of $f'_r\times g'$ is different from that of $f_r\times g$ only in a small neighborhood $\partial\left([-1,1]^2\right)$. 

Since the spines of the two splitting surfaces are disjoint except where they lie on $\partial X_K$, the difference between the graphics of $f'_r\times g'$ and $f_r\times g$ is the sequence of births and deaths that occur within $\epsilon$ of the boundary of the square, none of which share either $s$ or $t$ coordinates. By pulling these cusps back to the boundary of the square, we recover the graphic of $f_r\times g$. 

Now, every $f'_r$ is Morse, so there is an open, convex neighborhood of $f'_r\in N_r\subset C^\infty(X_K,\R)$ so that every $\varphi\in N_r$ is isotopic to $f'_r$. We can cover $\{f'_r\}$ by such neighborhoods, and since $\{f'_r\}\cong[0,1]$ is compact, there is a finite cover by such neighborhoods, so the path $\{f'_r\}$ can be replaced by a piecewise-linear path with each line segment contained in one convex neighborhood, and consecutive vertices $\varphi_0,...,\varphi_n\in\{f'_r\}$. For $\alpha\in[0,1]$, let $\beta=\frac{\alpha}{1-\alpha}$ and $\pi_\beta:\R^2\times\R^1$ be the projection of the $\R^2$ factor onto a line through the origin which has slope $\beta$. The graphic of the function $\left((1-\alpha)\varphi_0+\alpha\varphi_1\right)\times g'$ is the graphic of $\pi_\beta\circ\left(\varphi_0\times\varphi_1\times g'\right):X_K\rightarrow\R^1\times\R^1$.

The maps can all be chosen so that $\varphi_i$, $\varphi_{i+1}$, $g'$, $\varphi_i\times g'$, $\varphi_{i+1}\times g'$ and $\varphi_i\times\varphi_{i+1}$ are all stable, meaning each has an open neighborhood of isotopic Morse functions in the appropriate vector space. Since a projection is a continuous map, the preimage of each of these open neighborhoods in $C^\infty(X_K, \R^3)$ under projection is an open neighborhood of $\varphi_i\times\varphi_{i+1}\times g'$, and their intersection is also open. Since stable functions between 3-manifolds are dense in the $C^\infty$ topology, as shown by Mather \cite{mather}, we can ensure that $\varphi_i\times\varphi_{i+1}\times g'$ will be stable, as will be the maps in the projections listed above. Then, by Mather's classification of singularities of stable maps between 3-manifolds \cite{mather}, the discriminant set of $\varphi_i\times\varphi_{i+1}\times g'$ is a compact 2-submanifold $S\subset X_K$, and its image is an immersed 2-manifold with cusps.

Let $F=\varphi_{i}\times\varphi_{i+1}\times g'$. For $p\in S$, if $F(p)$ is not a cusp, there is a map $T_pS\rightarrow T_{F(p)}F(S)$. If this plane is parallel to the $s_1-s_2$ plane, this means that the tangency between $\varphi_i^{-1}(F(p))$ and $\varphi^{-1}_{i+1}(F(p))$ is preserved under two degrees of freedom independent of the parameter of $g'$, which means $p$ is at a vertex of the graphic of $\varphi_i\times\varphi_{i+1}$. As this graphic is generic, there are only finitely many such points. When $T_{F(p)}F(S)$ is not parallel to the $s_1-s_2$ plane, the intersection of $T_{F(p)}F(S)$ with the plane $t=g'(p)$ determines a slope in $\R^2$. For each $p\in S$ where $F(p)$ is not a cusp and $T_{F(p)}F(S)$ is not parallel to $t= g'(p)$, let $\sigma(p)$ be this slope. Observe that, though the number of non-cusp points where the slope is not well-defined is finite, the Intermediate Value Theorm implies that any connected level curve of $\sigma^{-1}(\beta)$ must contain a nonzero, even number of points $p$ at which $T_{F(p)}F(S)$ is parallel to $t=g'(p)$. 

We may perturb $F$ slightly to ensure that $\sigma$ is a Morse function on $S$ away from the finite number of points and curves (corresponding to the vertices of the graphic $\varphi_i\times\varphi_{i+1}$, and cusps of $F(S)$). This allows us to locally identify a patch of $F(S)$ with the graph of some function $\gamma:\R^2\rightarrow\R$. If $x=\gamma(y,z)$, and for some fixed $z=g'(p)$, $\frac{dx}{dy}=\sigma(F^{-1}(x,y,z))$, we have that there is some smooth $\gamma':\R\rightarrow\R$ so that ${\displaystyle \gamma(y,z)=\gamma'(z)+\int_{(x,0,z)}^{(x,y,z)}\sigma\circ F^{-1}dy}$. 

The discriminant set of $\pi_\beta\circ F$ is the image under $F$ of the closure of the set of points $\sigma^{-1}(\beta)$ in $S$. 

Since $\sigma$ is Morse, all but finitely many $\beta$ are regular values. A preimage of a regular $\beta$ value is a 1-dimensional submanifold of $S$. When $\beta$ is passing through a regular value, points in the interior of $\sigma^{-1}(\beta)$ map to the interior of an edge of the graphic of $pi_\beta\circ F$. If $p\in\overline{\sigma^{-1}(p)}\setminus\sigma^{-1}(p)$, 
%restricting $\gamma$ to points in the interior of $\sigma^{-1}(\beta)$ gives that such a point gets mapped to the interior of an edge of the graphic of $\pi_\beta\circ F$. If the point $p$ is not in the interior of $\sigma^{-1}(\beta)$, then it is in the closure of this set, so 
then it is either a point at which $T_{F(p)}F(S)$ is parallel to $t=g(p)$, or $F(p)$ belongs to a cusp of $F(S)$. In the first case, the Intermediate Value Theorem implies that such points must come in pairs. If no such points are introduced or removed at this level set $\beta$, then such points get mapped to the interior of an edge, as well. If a pair of such points is introduced or removed, the graphic of $\pi_\beta\circ F$ gains or loses a pair of cusps connected by an edge in a ``dovetail'' type move. In the case that $F(p)$ is a cusp, $p$ gets sent to a cusp of the graphic of $\pi_\beta\circ F$. Generically, the preimage of cusp points of $F(S)$ is transverse to level sets of $\sigma$. When $\beta$ is regular and $\sigma^{-1}(\beta)$ is tangent to an arc of cusp, a pair of cusps is introduced in the graphic of $\pi_\beta\circ F$ is either another dovetail move or an eye. 

\begin{figure}[h!]
\def\svgwidth{0.2\linewidth}
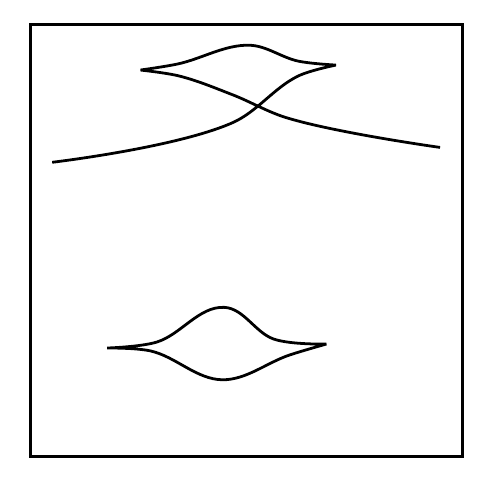
\end{figure}

When $\beta$ passes through a critical level of $\sigma$, the index of the critical level is either 0, 1 or 2. If it is 0 or 2, then there must be a pair of tangencies to $\R^2\times\{g'(p)\}$ immediately after or before, respectively, which do not exist before and after, respectively, this value of $\beta$, and in the graphic of $\pi_\beta\circ F$ these two corresponding cusps cancel each other out in an ``eye,'' during which process two vertices occupy the same level. If the critical point has index 1, a pair of plane-parallel points cancel and another pair is born, leading to an ``exchange'' type move at which there are two valence-two vertices at the same level. 

Thus there are only finitely many values of $\beta$ (and so, of $\alpha$ as well) at which $\left((1-\alpha)\varphi_0+\alpha\varphi_1\right)\times g'$ has a critical point which fails to be stable. In between these critical points, the graphic changes by some homotopy of the image of the discriminant set, which can be done generically except at a finite number of points introducing a triple point, self-tangency or double cusp points. 
\end{proof}

\section{Common stabilizations of opposite-sided tubed Heegaard surfaces}

To use the above results, we must understand how tubed Heegaard surfaces span a bridge sphere $S$. 

\begin{lem}If $K$ lies below (resp. above) $\Sigma$, then there are sweep-outs $f$ and $g$ for $\Sigma$ and $S$ respectively such that $\Sigma$ spans $S$ positively (resp. negatively).\label{lem-above-span-pos}
\end{lem}
\begin{proof}
First, note that if $d(S)>2n$ and $\Sigma$ is a genus-$n$ surface, Lemmas \ref{lem-span-pos-neg} and \ref{lem-split} guarantee that $\Sigma$ spans $S$ either positively or negatively, but not with both signs. We take the convention that the sweep-outs are oriented so that $\partial M\subset f^{-1}(-1)$, and $g^{-1}(-1)$ is the spine on the minus side of $S$. 

If $K$ lies below $\Sigma$, then for any value of $s\in(-1,1)$, there is a spine of $S$ on the minus side of $\Sigma_s$ which is entirely disjoint from $\Sigma_s$, and therefore there is a value $t^-$ close to $-1$ and a sweep-out $g$ so that $S_{t_-}\subset f^{-1}([-1,s])$, so $S_{t_-}$ is mostly below (in fact, entirely below) $\Sigma_s$. Fix one such $s$. Since the graphic must span either positively or negatively, there must be another $t_+>t_-$ such that $S_{t_+}$ is mostly above $\Sigma_s$, and so $\Sigma$ spans $S$ positively. One way to construct this pair $(s,t_+)$ explicitly is to take $s$ very close to $-1$, so that $\Sigma_s$ is very close to $\partial M$, and then for some $t_+$ it is clear that $S_{t_+}\cap f^{-1}([-1,s])$ is a collection of once-punctured disks in $S_{t_+}$. 

A symmetric argument shows that if $K$ lies above $\Sigma$, then there are sweep-outs so that $\Sigma$ spans $S$ negatively.
\end{proof}

Furthermore, every pair of sweep-outs $f$ and $g$ for $\Sigma$ and $S$ define a graphic in which $\Sigma$ spans $S$ with the same sign, independent of the choice of sweep-outs.

\begin{lem}
Let $K$ be an $n$-bridge knot with $n\geq 3$, $d(K)>2n$, and $\Sigma$ a Heegaard surface for $X_K$. If there are sweep-outs $f, f'$ corresponding to $\Sigma$ such that $f$ spans $g$ positively and $f'$ spans $g$ negatively, then during any isotopy $f_r$ with $f_0=f$ and $f_1=f'$, there is some $r$ such that $f_r$ must either split $g$ or span $g$ with both signs.
\label{lem-isotopy-span}
\end{lem}
\begin{proof}
Since $f_0$ and $f_1$ are given by equivalent surfaces, there is a one-parameter family of sweep-outs $\{f_r\}$, and by Lemma \ref{lem-isotopy} we know that the graphic $f_r\times g$ is generic except at finitely many values of $r\in[0,1]$, where it has two valence-two or valence-four vertices occupying the same level, or a single valence-6 vertex. 

Assume that away from the non-generic $r$, $f_r$ neither spans nor splits $g$ with both signs. Therefore $f_r$ splits $g$ either positively or negatively, but not both. Then there is some non-generic $r_0$ such that $f_{r_0-\epsilon}$ splits $g$ positively and $f_{r_0+\epsilon}$ spans $g$ negatively. At this $r_0$, the regions where $S$ is mostly above and mostly below $\Sigma$ must exchange, so there is a vertex $(s,t)$ in the graphic of $f_{r_0}\times g$ so that $S_t$ is, say, mostly above $\Sigma_{s-\epsilon}$ and mostly below $\Sigma_{s+\epsilon}$. Then $f_{r_0}|_{S_{t+\epsilon}}$ is a Morse function on $S_{t+\epsilon}$. If $(s,t)$ is a valence-4 vertex, this Morse function has exactly two index-1 critical points, and the rest are index 0 or 2. Then $S$ admits a decomposition into a collection of disks and once-punctured disks, which have Euler characteristic at least 0, connected by two bands of Euler characteristic $-1$, so $\chi(S)\geq -2$. If $(s,t)$ is a valence-6 vertex, then a similar argument shows that $\chi(S)\geq-3$.  But if $n\geq 3$, we have that $\chi(S)\leq-4$, a contradiction.

Therefore for some generic $r_0$, $f_{r_0}$ either spans $g$ both positively and negatively, or splits $g$.  
\end{proof}

\begin{cor}If $\Sigma$ is a tubed Heegaard surface with $K$ below (resp. above) $\Sigma$, $b(K)\geq 3$ and $d(K)>2b(K)$, then for any sweep-outs $f$ of $\Sigma$ and $g$ of $S$ with $\partial M\subset f^{-1}(-1)$ and $g^{-1}(-1)$ on the minus side of $S$, $f$ spans $g$ positively (resp. negatively).\label{cor-isotopy-span}\end{cor}
\begin{proof}
If there are two graphics, one in which $f$ spans $g$ positively, and the other in which $f$ spans $g$ negatively, by Lemma \ref{lem-isotopy-span} there is an isotopy between them during which the graphic will either split or span with both signs. But $g(\Sigma)=n$, and $n<2n-1$ so the graphic cannot span with both signs, and $n<\frac{1}{2}d(S)$ so the graphic cannot split. 
\end{proof}
\begin{cor}
Let $K$ be an $n$-bridge knot with $n\geq 3$ and $d(K)>2n$. Let $\Sigma$ be a tubed Heegaard surface so that $K$ lies below $\Sigma$, and $\Sigma'$ a tubed Heegaard surface so that $K$ lies above $\Sigma'$. Then a common stabilization of $\Sigma$ and $\Sigma'$ has genus at least $\min\{\frac{1}{2}d(S),2n-1\}$. \label{cor-common-stab}
\end{cor}
\begin{proof}
Let $f''_0$ be the sweep-out given by $\Sigma''$ as a stabilization of $\Sigma$, so that by Lemma \ref{lem-stab} $f''_0$ spans $g$ positively. Let $f''_1$ be the sweep-out given by $\Sigma''$ as a stabilization of $\Sigma'$, so that $f''_1$ spans $g$ negatively. By Lemma \ref{lem-isotopy-span}, at some point during the isotopy between these, the graphic either spans with both signs or splits. The conclusion follows.
\end{proof} 

\begin{cor}If $d(S)\geq 4n$, then $g(\Sigma'')\geq2n-1$.\label{cor-4n}\end{cor}

This concludes the proof of Theorem \ref{mainthm}

%\section{Distinct tubed Heegaard Splittings}
%
%In this section, we complete the proof of Theorem \ref{mainthm} by showing that, given the assumptions, the ${2n\choose n}$ surfaces are all distinct. 
%
%\begin{lem}
%For $K$ an $n$-bridge knot, let $\Sigma_0$, $\Sigma_1$ be two tubed Heegaard splittings of $X_K$ with $K$ on the same side. If $\Sigma_0$ and $\Sigma_1$ are constructed from different indices as in the proof of Lemma \ref{lem-2n-choose-n}, but are isotopic, then $d(K)\leq 2n$.
%\label{lem-S2-cross-I}\end{lem}
%\begin{proof}
%Assume there is an isotopy $\Sigma_r$ between the two surfaces, and $d(K)>2n$. Let $v_+$ and $v_-$ be the central vertices of the spines of the bridge sphere, $g^{-1}(1)$ and $g^{-1}(-1)$ respectively. We may assume that, in general, for no $r$ does the surface $\Sigma_r$ contain both of these points. Assume that for some $r$, $\Sigma_r$ does intersect one of these vertices, say $v_-$. Let $e$ be an edge adjacent to $v_-$, so that there is some $\epsilon>0$ such that $|\Sigma_{r-\epsilon}\cap e|-|\Sigma_{r+\epsilon}\cap e|=\pm1$. Therefore the parity of the intersection number is changed, so by the same arguments as in Lemma \ref{lem-isotopy-span} and Corollary \ref{cor-isotopy-span}, no such $r$ can exist. 
%
%Therefore, the isotopy $\Sigma_r$ can be performed to avoid small open neighborhoods of $v_\pm$, i.e. it can be performed in a copy of $S^2\times I-K$, where the bridge sphere from which $\Sigma_0$ and $\Sigma_1$ are obtained is isotopic to $S^2\times\{\frac{1}{2}\}$. 
%
%
%\end{proof}

\section{The tunnel systems}
By choosing a nonseparating system of $n-1$ compression disks for the compression body containing $K$, the cocores of these 2-handles can be extended to a tunnel system for $K$. In this section, we give a canonical way to make this choice

\begin{dfn}
Let $A_i$ be an annulus of a tubed Heegaard splitting which runs along $K$ and has left foot at $x_i$ and right foot at $x_{r(i)}$. If there is no annulus running along $K$, $A_j$, with length $\ell>(i-j\mod2n)$, then $A_i$ is \emph{adjacent to $K$}.
\end{dfn}

In other words, $A_i$ is adjacent to $K$ if no other annulus runs between $A_i$ and $K$. Then any annulus with left foot between $i$ and $r(i)$ also has its right foot in this range. Let all such annuli be \emph{surrounding} $A_i$.

\begin{dfn}
Let $\Sigma$ be a tubed Heegaard surface. An annulus of $\Sigma$ which is adjacent to $K$, together with all the annuli which surround this one, is called a \emph{chunk} of $\Sigma$. The \emph{size} of the chunk is the total number of annuli. The annulus which is adjacent to $K$ is the \emph{defining annulus} for the chunk.
\end{dfn}
The splittings $\Sigma_a$ and $\Sigma_b$ have $n$ chunks all of size 1, the $3$-bridge splitting with index $(1,2,3)$ has a single chunk of size 3, and the $7$-bridge splitting with index $(1,3,5,6,7,11,12)$ has four chunks of sizes $1, 1, 3$ and $2$. We now show how to associate $k-1$ tunnels to each chunk of size $k$. 

Assume without loss of generality that $K$ is below $\Sigma$. Annuli of $\Sigma$ which have their left feet at odd values of $i$ will be said to go \emph{up}, while those with even-numbered left food indices will be said to go \emph{down}. Since $K$ is below $\Sigma$, any annulus which is adjacent to $K$ must go up. Let there be a chunk of size $k$, with $A_i$ being the defining annulus. Let $\gamma_i$ be the bridge arc of $K$ between $x_i$ and $x_{i+1}$. The $k-1$ tunnels associated to this chunk, $\tau^i_1,...,\tau^i_k$ will all have $\partial\tau^i_j\subset \gamma_{i-1}$. 

Let $A_j$ be an annulus of this chunk which goes down. If $A_j$ has length 1, let $D$ be the bridge disk associated to $\gamma_j$. Let $W_1$ be the compression body which contains $K$. The component of $D\cap W_1$ which meets $S-\eta(K)$ is a non-separating compression disk for $\Sigma$ contained in $W_1$. A tunnel for $K$ which is dual to this compression disk is an arc with endpoints on $\gamma_{i-1}$ which intersects this disk in one point and can be levelled into $\Sigma$. Such an arc is shown in Figure \ref{fig-chunk-tunnels-1}.

\begin{figure}[h!]
\includegraphics[width=0.5\linewidth]{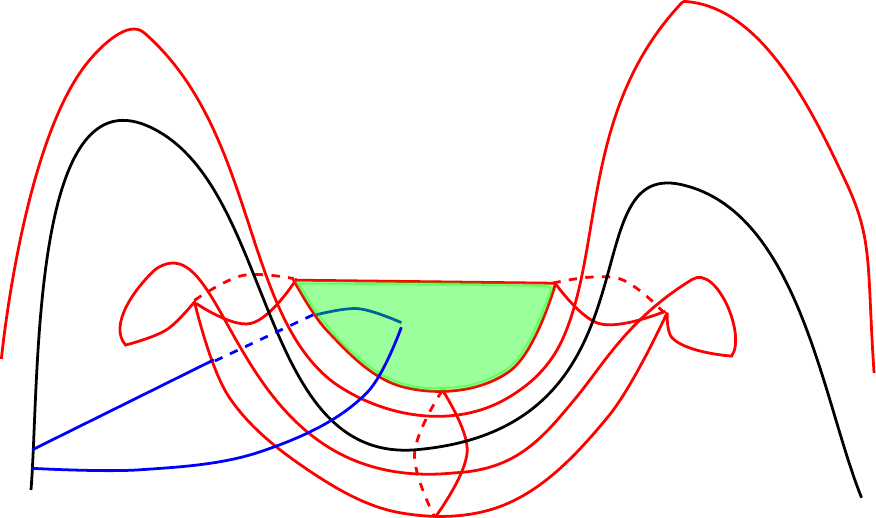}
\caption{A tunnel (blue) which is dual to a compression disk (green) which corresponds to a down-annulus of length 1 when $K$ is below $\Sigma$.}
\label{fig-chunk-tunnels-1}
\end{figure}

If $A_j$ is a down-annulus with length greater than 1, then we construct a compression disk for $\Sigma$ on the same side as $K$ by banding together the two bridge disks associated to $\gamma_j$ and $\gamma_{r(j)-1}$ with a strip which follows $A_j$, as shown in Figure \ref{fig-chunk-tunnels-2}.

\begin{figure}[h!]
\includegraphics[width=0.7\linewidth]{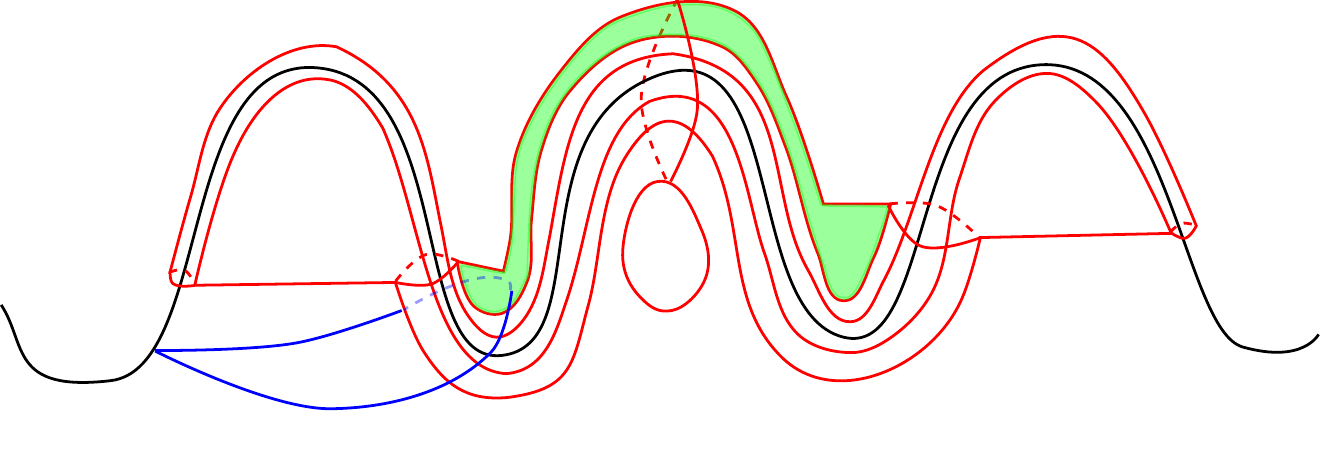}
\caption{Banding together two bridge disks along a down-annulus of length 3 to make a compression disk for $\Sigma$ whose dual is a tunnel.}
\label{fig-chunk-tunnels-2}
\end{figure}

If $A_k$ is an up-annulus which is \emph{not} the defining annulus $A_i$, we make use of a non-separating compression disk for $\Sigma$ whose boundary separates the ends of $A_k$. We choose this disk so that $\partial D\subset S$ separates the punctures $\left\{x_{r(k)},x_{r(k)+1},...,x_{r(i)-1}\right\}$ from the others. The arc which starts at $\gamma_{i-1}$, pierces this disk once, goes through $A_k$ and returns to $\gamma_{i-1}$ is parallel into $\Sigma$ (see Figure \ref{fig-chunk-tunnels-3}).

\begin{figure}[h!]
\includegraphics[width=0.7\linewidth]{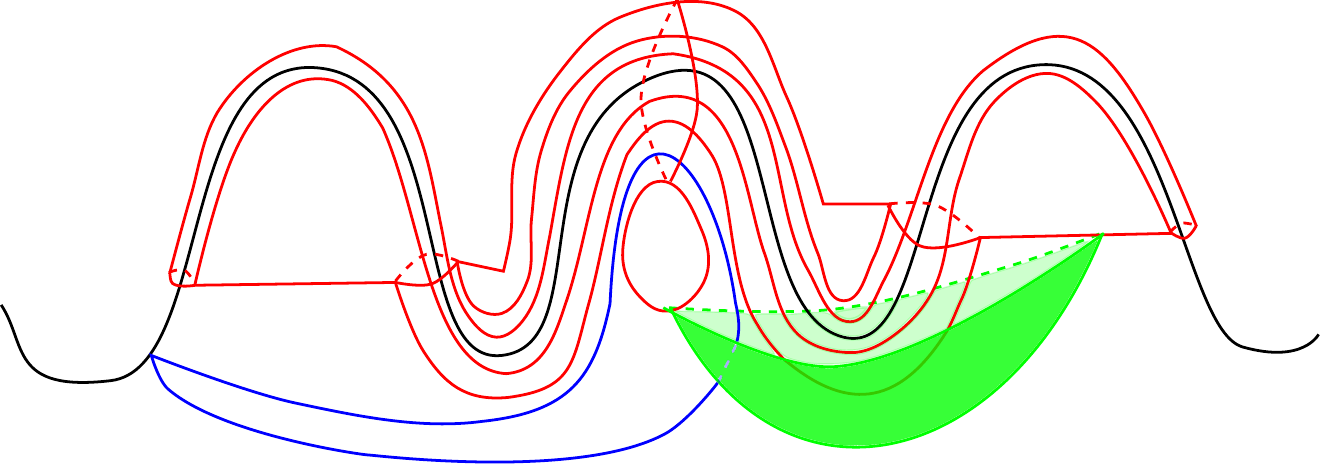}
\caption{Constructing a tunnel corresponding to an up-tube in a chunk when $K$ is below $\Sigma$.} 
\label{fig-chunk-tunnels-3}
\end{figure}

In this way, assigning tunnels dual to banded bridge disks for down-tubes, and dual to nonseparating compression disks for up-tubes, every chunk of size $k$ contributes $k-1$ tunnels, since all annuli but the defining annulus contributes exactly one tunnel. The remainder of the tunnels connect adjacent chunks: if there is a chunk defined by the annulus $A_i$, then there is a tunnel $\tau_i$ with $\partial_-\tau_i\in\gamma_{i-1}$ and $\partial_+\tau_i\in\gamma_{r(i)}$. However, if we connect all adjacent chunks together, there will be $n$ tunnels total. One of these is dual to a compression disk for $\Sigma$ in $S^3$ which intersects $K$ in one point, so we simply choose which chunk to ignore. So that this choice is consistent for all partitions of $\{1,...,2n\}$ into chunks, we omit the tunnel $\tau_i$ which connects $\gamma_{i-1}$ to $\gamma_{r(i)}$ where $i$ is the smallest index over all defining annuli for $\Sigma$. 

In constructing these tunnels, several choices were made: the decision to use the structure of the chunks is a choice of a system of compression disks. Within this choice, the structure of each chunk as being given by up- and down-tunnels is once again a choice of compression disks. Since any two complete systems of compression disks for a given handlebody are related by a sequence of handle slides, the tunnel systems corresponding to these choices are equivalent after a sequence of edge slides. The choice of endpoints for the tunnels is also trivial modulo edge slides of the tunnels. Finally, the choice of which redundant chunk-connector tunnel is to be omitted is arbitrary, but also corresponds to a choice of system of compressing disks. Thus, any minimal tunnel system for $K$ is equivalent to a minimal tunnel system for $K$ constructed in this way, after a sequence of isotopies and edge slides. 

Figure \ref{fig-3-bridge-tunnels} shows how we construct the $\dfrac{1}{2}\displaystyle{6\choose3}=10$ two-tunnel tunnel systems of a $3$-bridge knot which correspond to a Heegaard surface which is above $K$. The horizontal line represents a bridge sphere, and the braids are omitted for simplicity. The other ten tunnel systems are symmetric to the given ones, on the other side of the bridge sphere. 
\begin{figure}[h!]
\subfigure[I=(1,3,5)]{\includegraphics[width=0.18\textwidth]{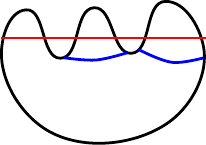}}
\subfigure[I=(1,3,4)]{\includegraphics[width=0.18\textwidth]{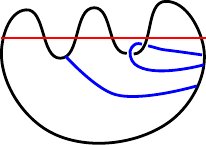}}
\subfigure[I=(1,2,5)]{\includegraphics[width=0.18\textwidth]{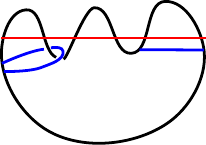}}
\subfigure[I=(3,5,6)]{\includegraphics[width=0.18\textwidth]{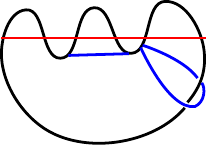}}
\subfigure[I=(1,2,4)]{\includegraphics[width=0.18\textwidth]{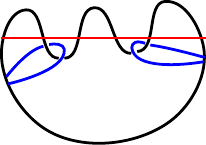}}
\\
\subfigure[I=(3,4,6)]{\includegraphics[width=0.18\textwidth]{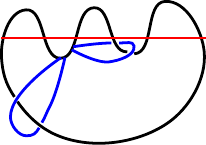}}
\subfigure[I=(2,5,6)]{\includegraphics[width=0.18\textwidth]{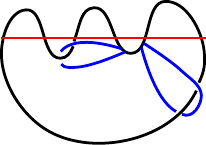}}
\subfigure[I=(1,5,6)]{\includegraphics[width=0.18\textwidth]{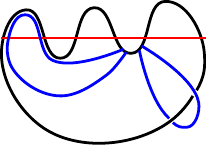}}
\subfigure[I=(1,2,3)]{\includegraphics[width=0.18\textwidth]{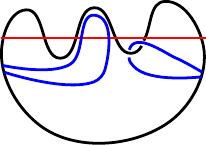}}
\subfigure[I=(3,4,5)]{\includegraphics[width=0.18\textwidth]{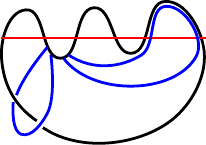}}
\caption{10 of the 20 tunnels of a 3-bridge knot.}
\label{fig-3-bridge-tunnels}
\end{figure}
\bibliographystyle{amsplain}
\bibliography{heegaard-high-distance-knot}
\end{document}